\documentclass[12pt]{article}

\usepackage{amsmath}
\usepackage{graphicx}
\usepackage{ upgreek }
\usepackage{ wasysym }
\usepackage{amsthm}
\usepackage{ amssymb }
\usepackage{ dsfont }
\usepackage{array}
\usepackage{multirow}
\usepackage{ gensymb }
\usepackage{subfig}
\usepackage[margin=1in]{geometry}
\usepackage{setspace}
\usepackage[T1]{fontenc}
\usepackage{mathptmx}
\usepackage{mathtools}
\usepackage{relsize}
\usepackage{placeins}

\usepackage[format=plain, font=footnotesize]{caption}
\newtheorem{theorem}{Theorem}

\newtheorem{definition}{Definition}
\newtheorem{lemma}{Lemma}

\newtheorem*{corollary}{Corollary}

\makeatletter
\newenvironment{rsmallmatrix}{\null\,\vcenter\bgroup
  \Let@\restore@math@cr\default@tag
  \baselineskip6\ex@ \lineskip1.5\ex@ \lineskiplimit\lineskip
  \ialign\bgroup\hfil$\m@th\scriptstyle##$&&\thickspace\hfil
  $\m@th\scriptstyle##$\crcr
}{%
  \crcr\egroup\egroup\,%
}
\makeatother

\newcommand{\appropto}{\mathrel{\vcenter{ 
  \offinterlineskip\halign{\hfil$##$\cr 
    \propto\cr\noalign{\kern2pt}\sim\cr\noalign{\kern-2pt}}}}} 
    
\title{Apollonian Equilateral Triangles}
\author{Christina Chen, Nan Li}
\date{December 29, 2012}

\begin{document}

\maketitle
\begin{abstract}
Given an equilateral triangle with $a$ the square of its side length and a point in its plane with $b$, $c$, $d$ the squares of the distances from the point to the vertices of the triangle, it can be computed that $a$, $b$, $c$, $d$ satisfy $3(a^2+b^2+c^2+d^2)=(a+b+c+d)^2$.  This paper derives properties of quadruples of nonnegative integers $(a,\, b,\, c,\, d)$, called triangle quadruples, satisfying this equation.  It is easy to verify that the operation generating $(a,\, b,\, c,\, a+b+c-d)$ from $(a,\, b,\, c,\, d)$ preserves this feature and that it and analogous ones for the other elements can be represented by four matrices.  We examine in detail the triangle group, the group with these operations as generators, and completely classify the orbits of quadruples with respect to the triangle group action.  We also compute the number of triangle quadruples generated after a certain number of operations and approximate the number of quadruples bounded by characteristics such as the maximal element.  Finally, we prove that the triangle group is a hyperbolic Coxeter group and derive information about the elements of triangle quadruples by invoking Lie groups.  We also generalize the problem to higher dimensions.
\end{abstract}
\section{Introduction}\FloatBarrier
The study of Apollonian circle packings is an elegant geometric subject that generates deep questions in the theory of hyperbolic Coxeter groups, Lie theory, and analytic number theory, and many interesting problems about these packings can be effectively
 solved with state-of-the-art techniques from these fields.  For example, \cite{Graham1}, \cite{Graham2}, and other papers in the same series by these authors derive a large number of properties of the Apollonian group and Apollonian quadruples from group theory, number theory, and geometry.  The case of Apollonian equilateral triangles is a similar but less well-studied topic, and the goal of the project was to answer natural questions about this case with advanced techniques in group theory and number theory.
 
Formally, an Apollonian circle packing is a fractal generated from triples of mutually tangent circles.  In such a packing, for any four mutually tangent circles, their curvatures, $a$, $b$, $c$, $d$, satisfy Descartes's equation, $2(a^2+b^2+c^2+d^2)=(a+b+c+d)^2$.  Note that for fixed $a$, $b$, $c$, exactly two values of $d$ satisfy the equation, generating two possible configurations, as illustrated in Figure~\ref{apolloniancase}.  Given the configuration in Figure~\ref{d}, additional circles can be inscribed in each lune successively, a process illustrated in Figure~\ref{apollonianoperation}.  In addition, it is clear that given four mutually tangent circles with curvatures $a$, $b$, $c$, $d$, the curvature of the new circle inscribed in the lune bounded by the circles with curvatures $a$, $b$, $c$ is $2(a+b+c)-d$, which is easily derived from Descartes's equation.

  \begin{figure}
  \centering
  \subfloat[If the fourth circle is internally tangent to the other three, then its curvature $d$ is negative.]{\label{d}\includegraphics[trim=-2.8cm 0cm -1.2cm 0cm,clip=true,totalheight=0.3\textwidth]{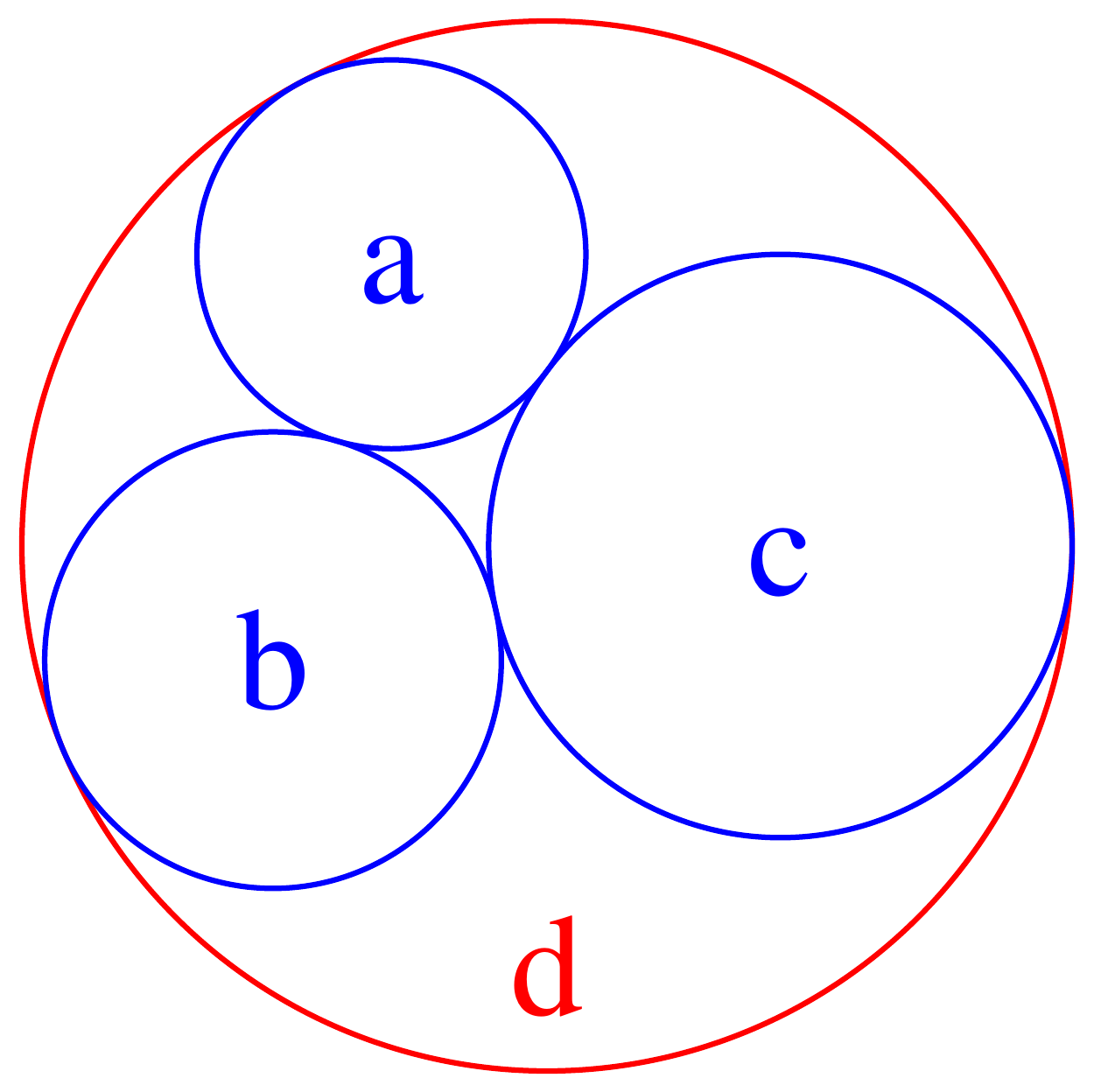}}
  ~\hspace{0.4cm}
  \subfloat[If the fourth circle is externally tangent to the other three, then its curvature $d'$ is positive.]{\label{d'}\includegraphics[trim=-1.5cm 0cm -0.7cm 0cm,clip=true,totalheight=0.3\textwidth]{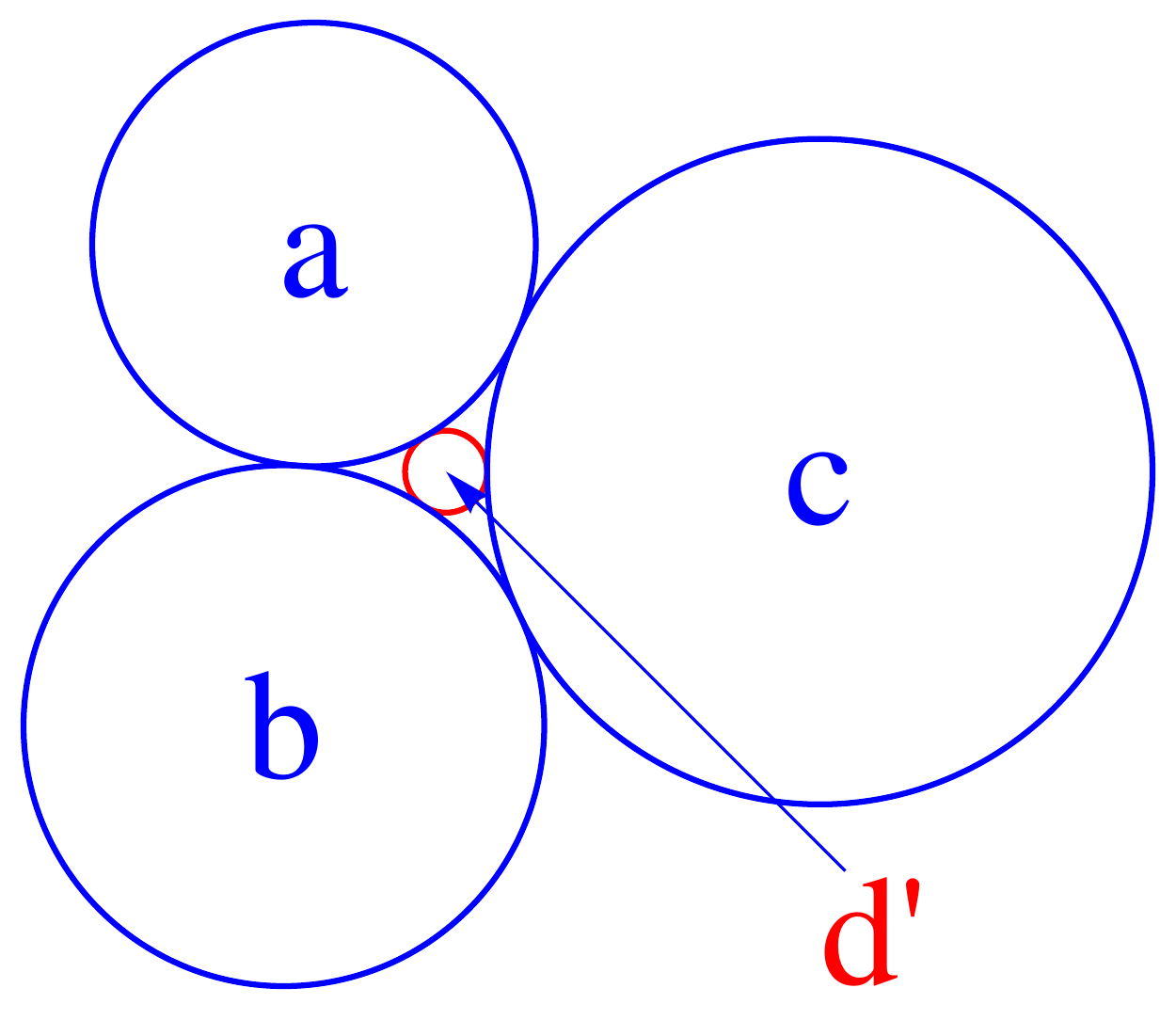}}
  \caption{There are two ways in which a fourth circle can be packed, and the two curvatures are related according to $d+d'=2(a+b+c)$.}
  \label{apolloniancase}
\end{figure}
  
  \begin{figure}
  \centering
  \subfloat{\includegraphics[trim=0cm -1cm 0cm 0cm,clip=true,totalheight=0.16\textwidth]{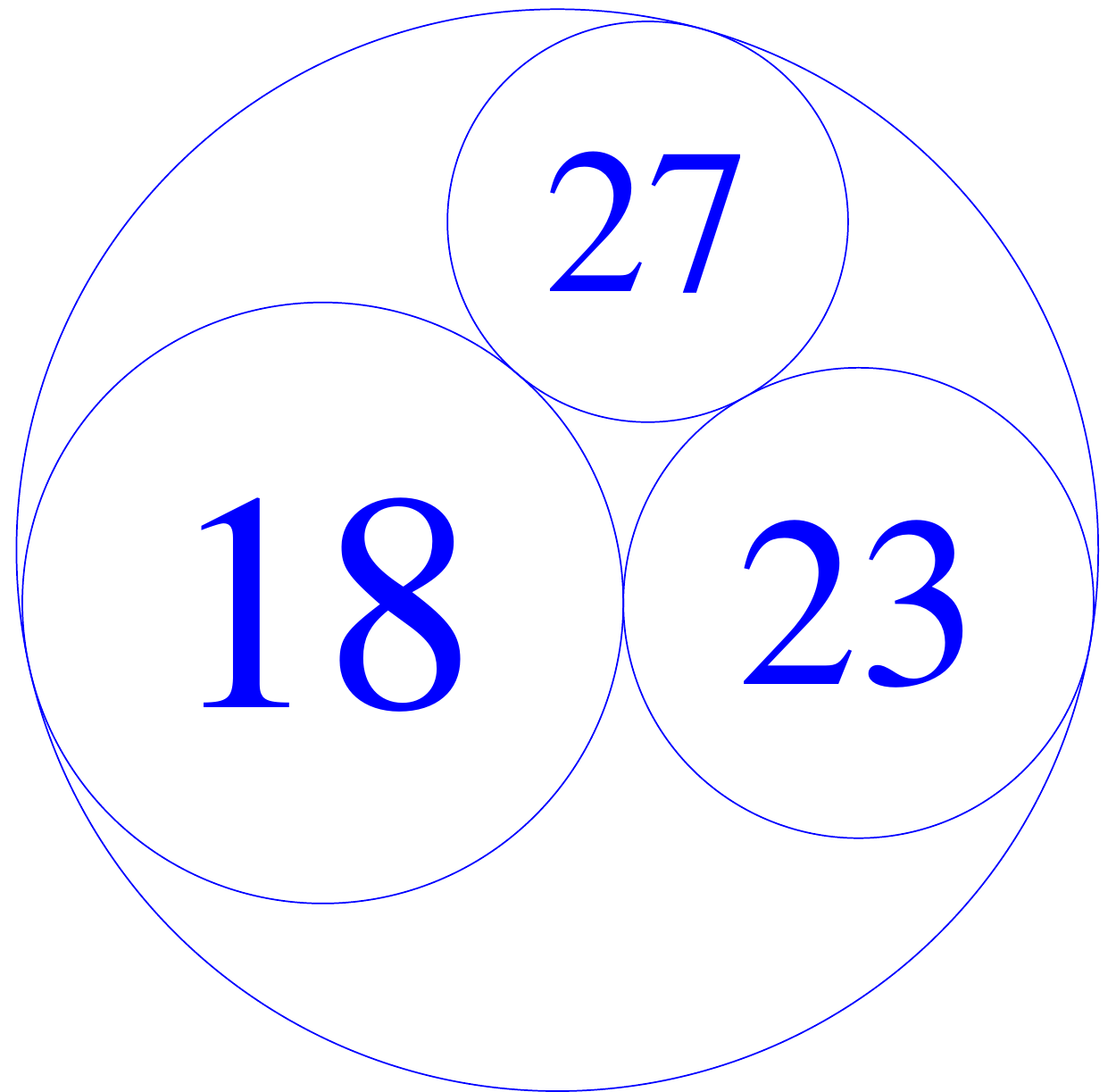}}
  ~\hspace{0.01mm}
  \subfloat{\includegraphics[trim=5cm 4cm 0cm 4cm,clip=true,totalheight=0.16\textwidth]{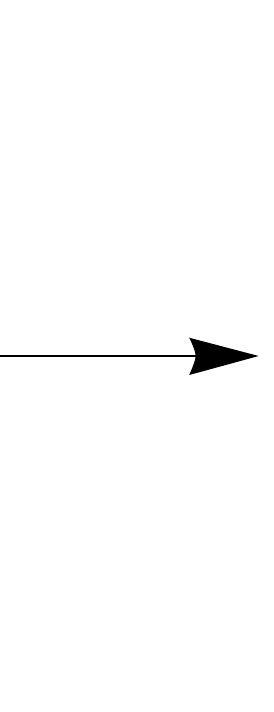}}
  ~\hspace{0.01mm}
  \subfloat{\includegraphics[trim=0cm -1cm 0cm 0cm,clip=true,totalheight=0.16\textwidth]{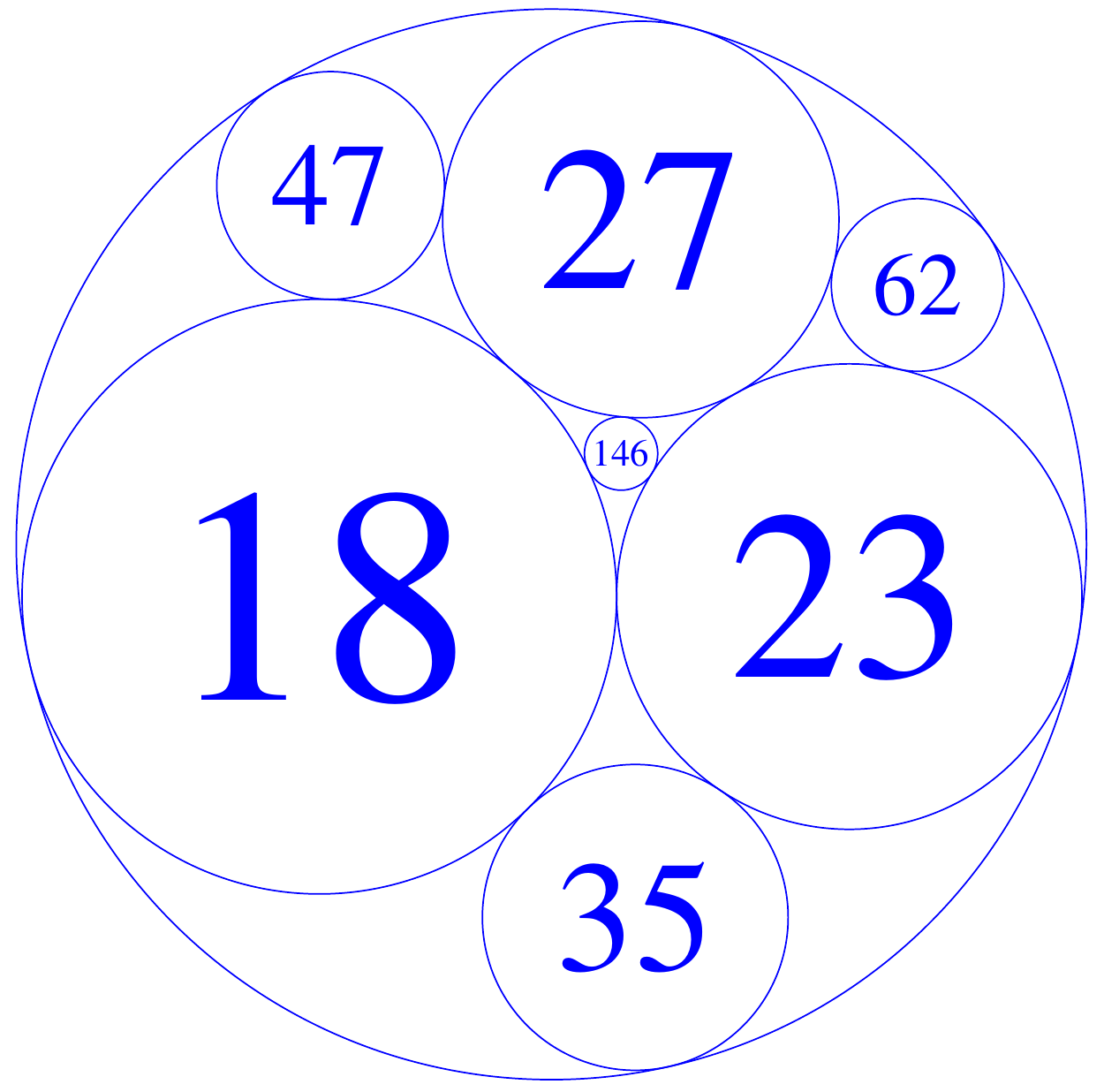}}
  ~\hspace{0.01mm}
  \subfloat{\includegraphics[trim=5cm 4cm 0cm 4cm,clip=true,totalheight=0.16\textwidth]{arrow3.pdf}}
  ~\hspace{0.01mm}
  \subfloat{\includegraphics[trim=0cm -1cm 0cm 0cm,clip=true,totalheight=0.16\textwidth]{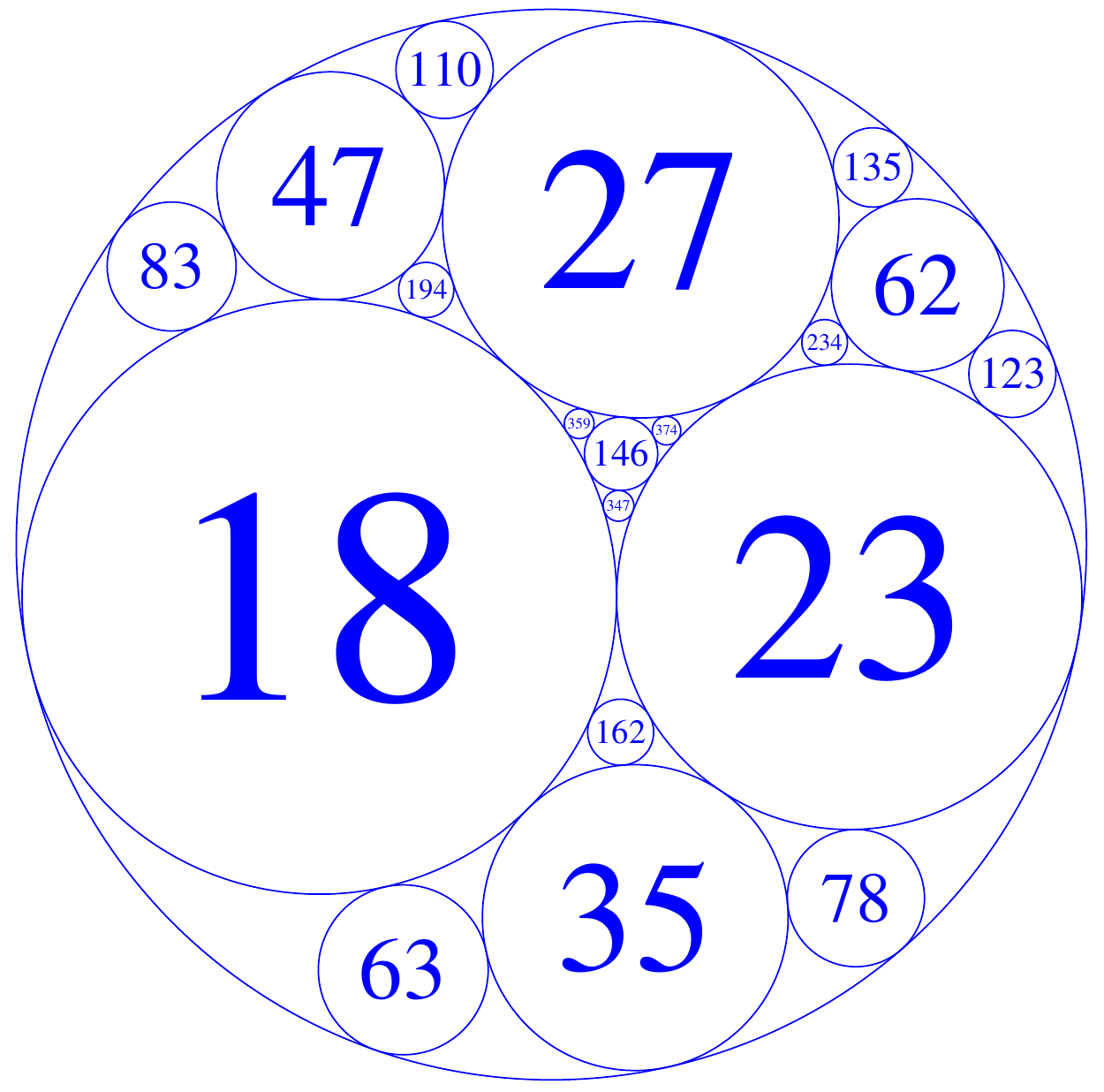}}~\hspace{0.01mm}
  \subfloat{\includegraphics[trim=5cm 4cm 0cm 4cm,clip=true,totalheight=0.16\textwidth]{arrow3.pdf}}
  ~\hspace{0.01mm}
  \subfloat{\includegraphics[trim=0cm -1cm 0cm 0cm,clip=true,totalheight=0.16\textwidth]{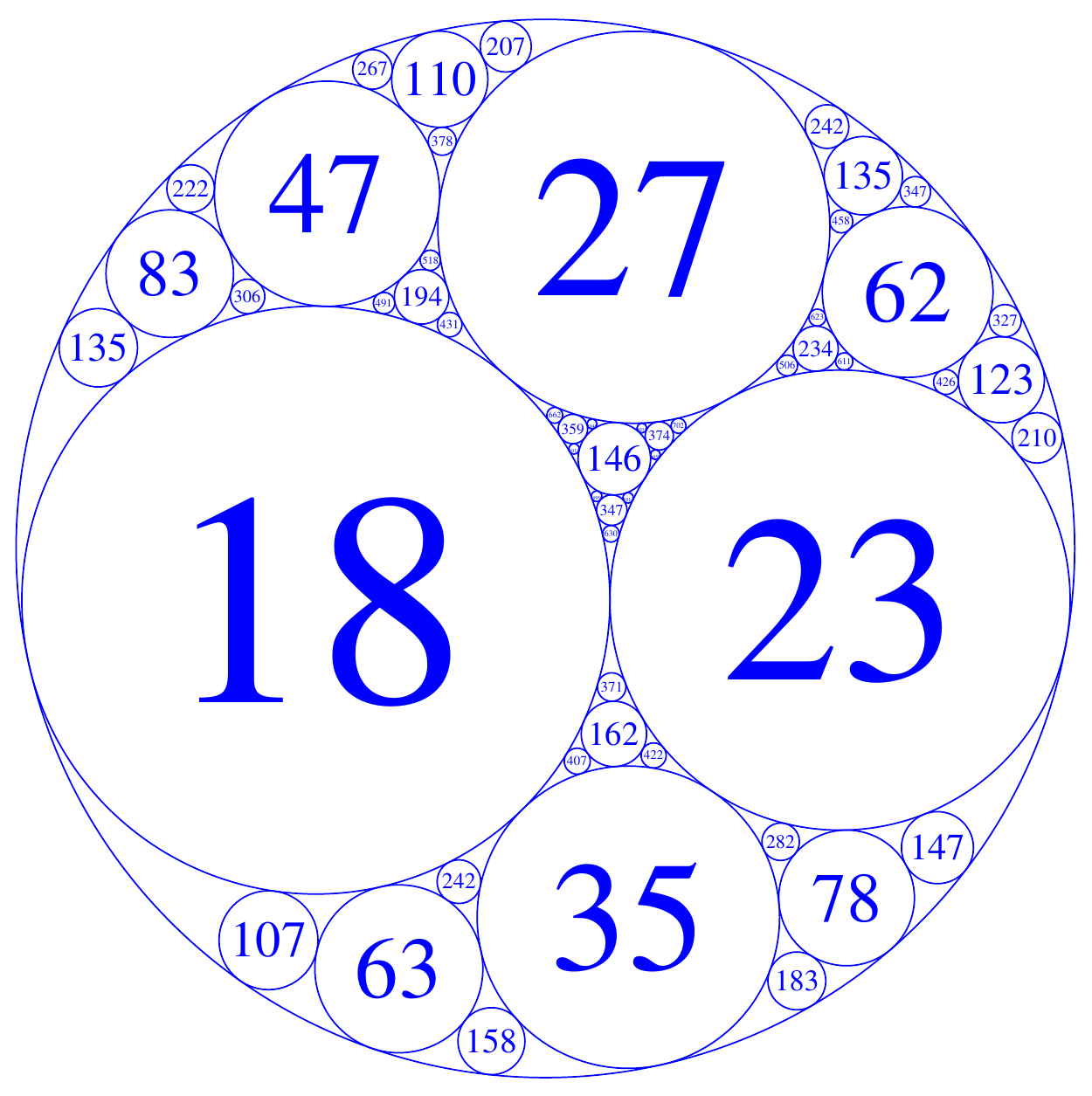}}
  \caption{At each stage, a circle is inscribed in each lune.  The curvature of the outer circle is $-10$.}
  \label{apollonianoperation}
\end{figure}

Define an Apollonian quadruple be a quadruple of nonnegative integers satisfying Descartes's equation.  Then the inscribing operation described above can be expressed as the transformation $(a,\, b,\, c,\, d)\rightarrow(a,\, b,\, c,\, 2(a+b+c)-d)$.  Now, consider the matrices $$
S_1=\begin{pmatrix}-1&2&2&2\\0&1&0&0\\0&0&1&0\\0&0&0&1\\\end{pmatrix},\hspace{1mm}
S_2=\begin{pmatrix}1&0&0&0\\2&-1&2&2\\0&0&1&0\\0&0&0&1\\\end{pmatrix},\hspace{1mm}
S_3=\begin{pmatrix}1&0&0&0\\0&1&0&0\\2&2&-1&2\\0&0&0&1\\\end{pmatrix},\hspace{1mm}
S_4=\begin{pmatrix}1&0&0&0\\0&1&0&0\\0&0&1&0\\2&2&2&-1\\\end{pmatrix}.
$$
It is clear that for an Apollonian quadruple $v=(a,\, b,\, c,\, d)^{T}$, we have $S_{4}v=(a,\, b,\, c,\, 2(a+b+c)-d)^{T}$, and analogous relations hold for $S_1$, $S_2$, $S_3$.  Therefore, $S_1$, $S_2$, $S_3$, $S_4$ correspond to transformations on $a$, $b$, $c$, $d$, respectively.  Define the Apollonian group to be the group with $S_1$, $S_2$, $S_3$, $S_4$ as generators.  Previous papers have proved that the Apollonian group is a Coxeter group.  
   
Now, consider an equilateral triangle $T$ and a point $P$ in its plane.  Let $a$ denote the square of the side length of $T$, and let $b$, $c$, and $d$ denote squares of the distances from $P$ to the vertices of $T$.  With simple trigonometry, we can derive the relation $3(a^{2}+b^{2}+c^{2}+d^{2})=(a+b+c+d)^{2}$, which is remarkably similar to Descartes's equation.  For the relation for the general case of an $n$-dimensional simplex, see \cite{Gregorac}.  

In the Apollonian triangle case, as in the Apollonian circle case, there exist four matrices corresponding to reflection operations (instead of inscribing operations), as illustrated in Figure~\ref{operation}, with which a triangle group can be defined.  Because of the great similarities between this case and the Apollonian case, we are motivated to verify whether many of the properties in the Apollonian case also hold for the triangle case, and where they do not, to determine whether the conclusions can be modified.  We also analyze properties of the triangle group and triangle quadruples that have not been studied in the Apollonian case.

In Section~\ref{trianglegroup}, we define the terms relevant to the triangle group and quadruples.  We also discuss the reflection operation in greater detail.  In Section~\ref{rootquadruples}, we introduce the algorithm of reducing triangle quadruples, which was also extensively analyzed for Apollonian packings.  In Section~\ref{geometricgroups}, we prove that the triangle group is a Coxeter group.  In Section~\ref{properties}, we compute several properties of triangle quadruples, many of which have not been determined for Apollonian quadruples, including the number of triangle quadruples bounded by height and by maximal element, as well as results about the product of elements in triangle quadruples.  In Section~\ref{higherdimensions}, we consider the problem in higher dimensions.  Finally, in Section~\ref{openquestions}, we collect open questions requiring additional research.

\section{Triangle Group}\label{trianglegroup}
\subsection{Basic Definitions}
\begin{definition}
A \emph{triangle quadruple} $t=(a,\, b,\, c,\, d)$ is a quadruple of nonnegative integers satisfying $$3(a^{2}+b^{2}+c^{2}+d^{2})=(a+b+c+d)^{2}.$$
\end{definition}
\begin{definition}
A triangle quadruple $(a,\, b,\, c,\, d)$ is \emph{primitive} if $\gcd(a,\, b,\, c,\, d)=1$.
\end{definition}
\begin{definition}
The \emph{triangle group} $T$ is the subgroup of $GL(4, \mathds{Z})$ generated by
$$
S_1=\begin{pmatrix}-1&1&1&1\\0&1&0&0\\0&0&1&0\\0&0&0&1\\\end{pmatrix},\hspace{1mm}
S_2=\begin{pmatrix}1&0&0&0\\1&-1&1&1\\0&0&1&0\\0&0&0&1\\\end{pmatrix},\hspace{1mm}
S_3=\begin{pmatrix}1&0&0&0\\0&1&0&0\\1&1&-1&1\\0&0&0&1\\\end{pmatrix},\hspace{1mm}
S_4=\begin{pmatrix}1&0&0&0\\0&1&0&0\\0&0&1&0\\1&1&1&-1\\\end{pmatrix}.
$$
\end{definition}
\subsection{General Notes About Operations}
Note that the operation generating the triangle quadruple $t'=(a,\, b,\, c,\, a+b+c-d)$ from the triangle quadruple $t=(a,\, b,\, c,\, d)$ can be characterized by $t'=S_4t$, where $t'$ and $t$ are written as column vectors.  Analogous relations hold for operations on the other elements.  

Geometrically, the operation generating $(7,\, 4,\, 9,\, 1)$ from $(7,\, 4,\, 3,\, 1)$ is illustrated in Figure~\ref{operation}, in the case where 7 corresponds to the square of the side length of the equilateral triangle.  However, if we remove the restrictions on order, then every quadruple actually generates four different configurations depending on the side length of the equilateral triangle.  They can be encapsulated by one diagram, as illustrated in Figure~\ref{encapsulation}.

\begin{figure}
  \centering
  \captionsetup[subfloat]{labelformat=empty, labelsep=none}
  \subfloat[$t=(7,\, 4,\, 3,\, 1)$]{\includegraphics[trim=0cm -5cm 0cm 0cm,clip=true,totalheight=0.3\textwidth]{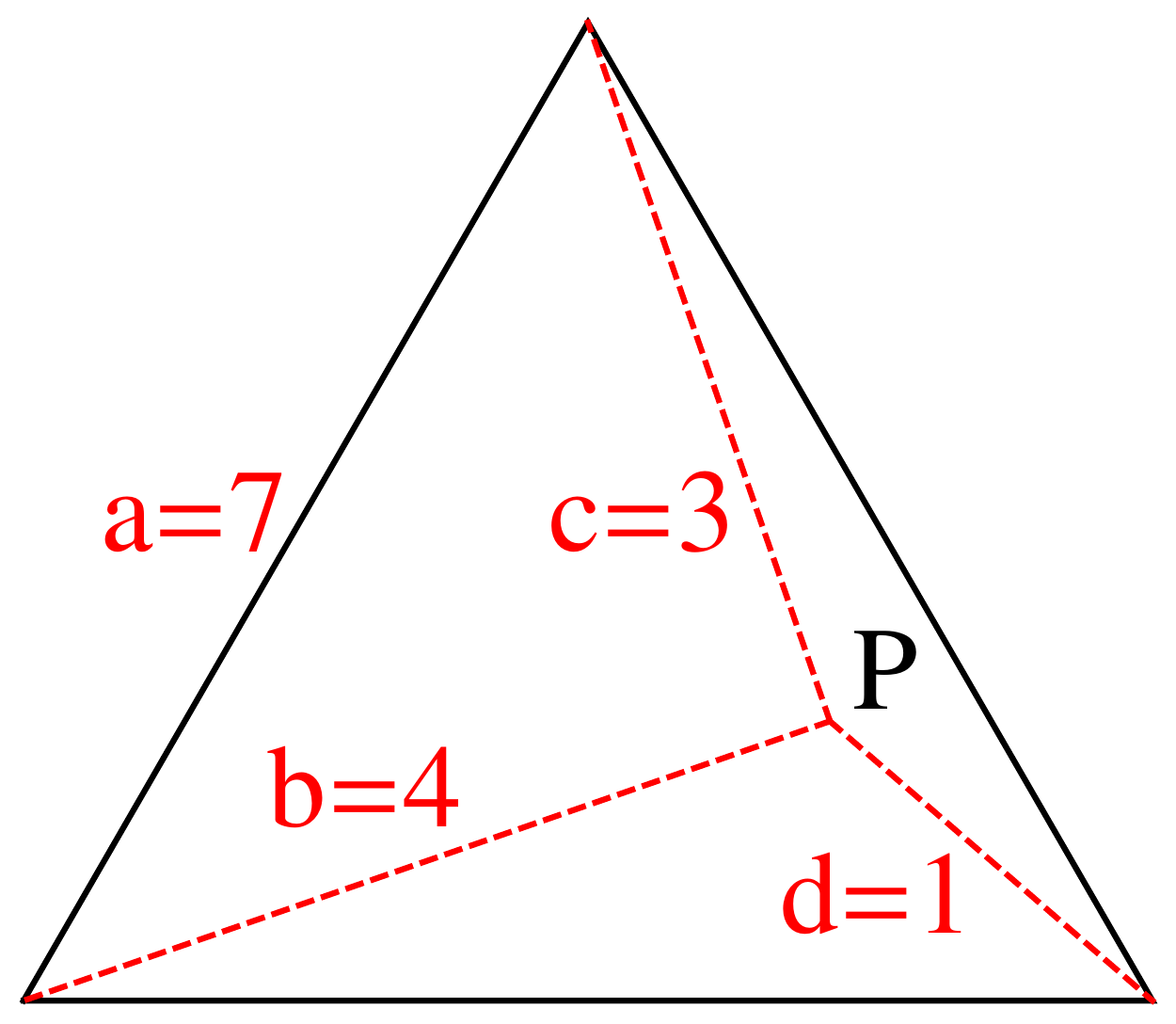}}
  ~\hspace{0.01mm}
  \subfloat{\includegraphics[trim=5cm 2cm 0cm 2cm,clip=true,totalheight=0.36\textwidth]{arrow3.pdf}}
  \subfloat{\includegraphics[trim=0cm -0.95cm 0cm 0cm,clip=true,totalheight=0.3\textwidth]{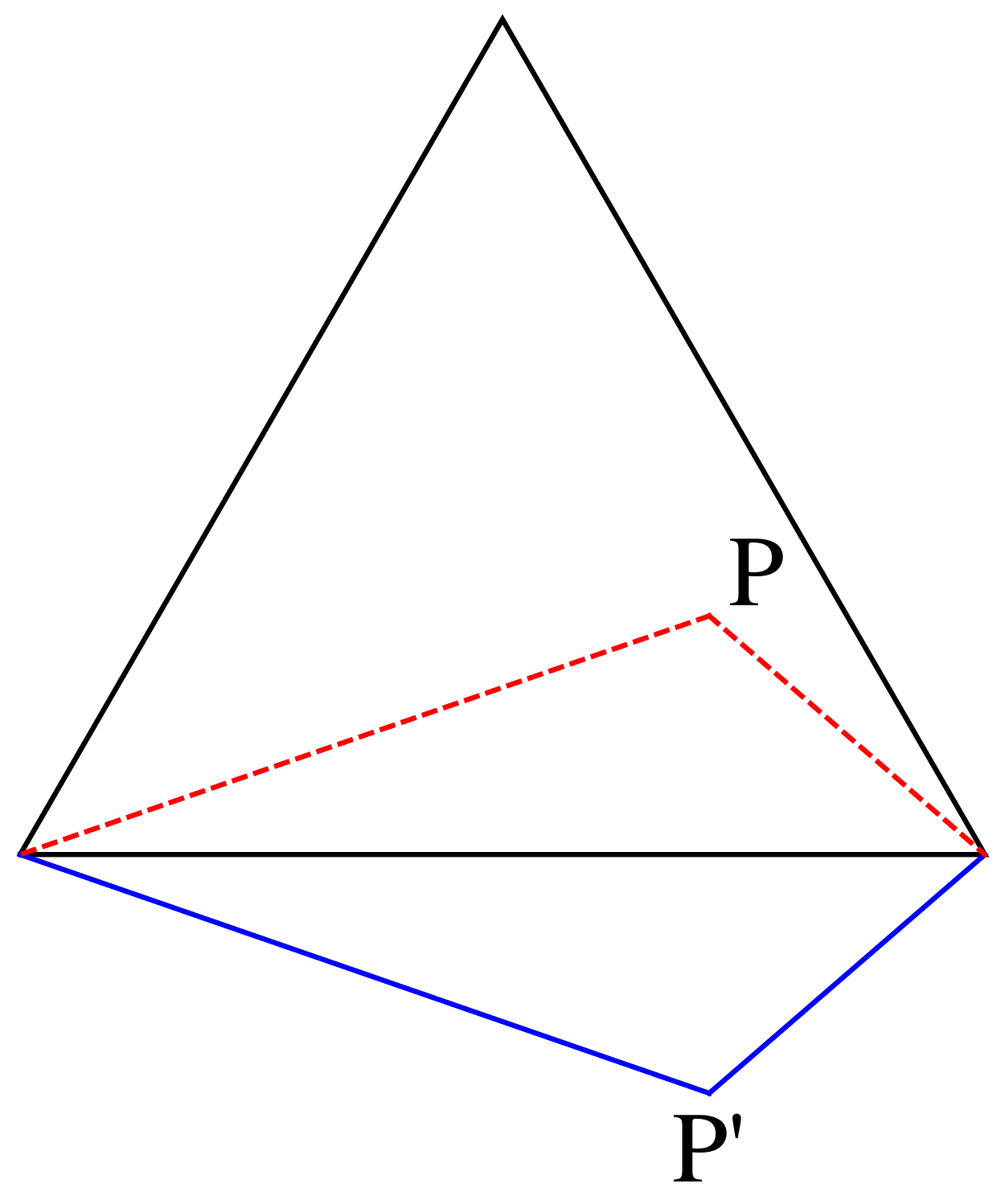}}
  \subfloat{\includegraphics[trim=5cm 2cm 0cm 2cm,clip=true,totalheight=0.36\textwidth]{arrow3.pdf}}
  \subfloat[$t'=(7,\, 4,\, 9,\, 1)$]{\includegraphics[trim=0cm -0.95cm 0cm 0cm,clip=true,totalheight=0.3\textwidth]{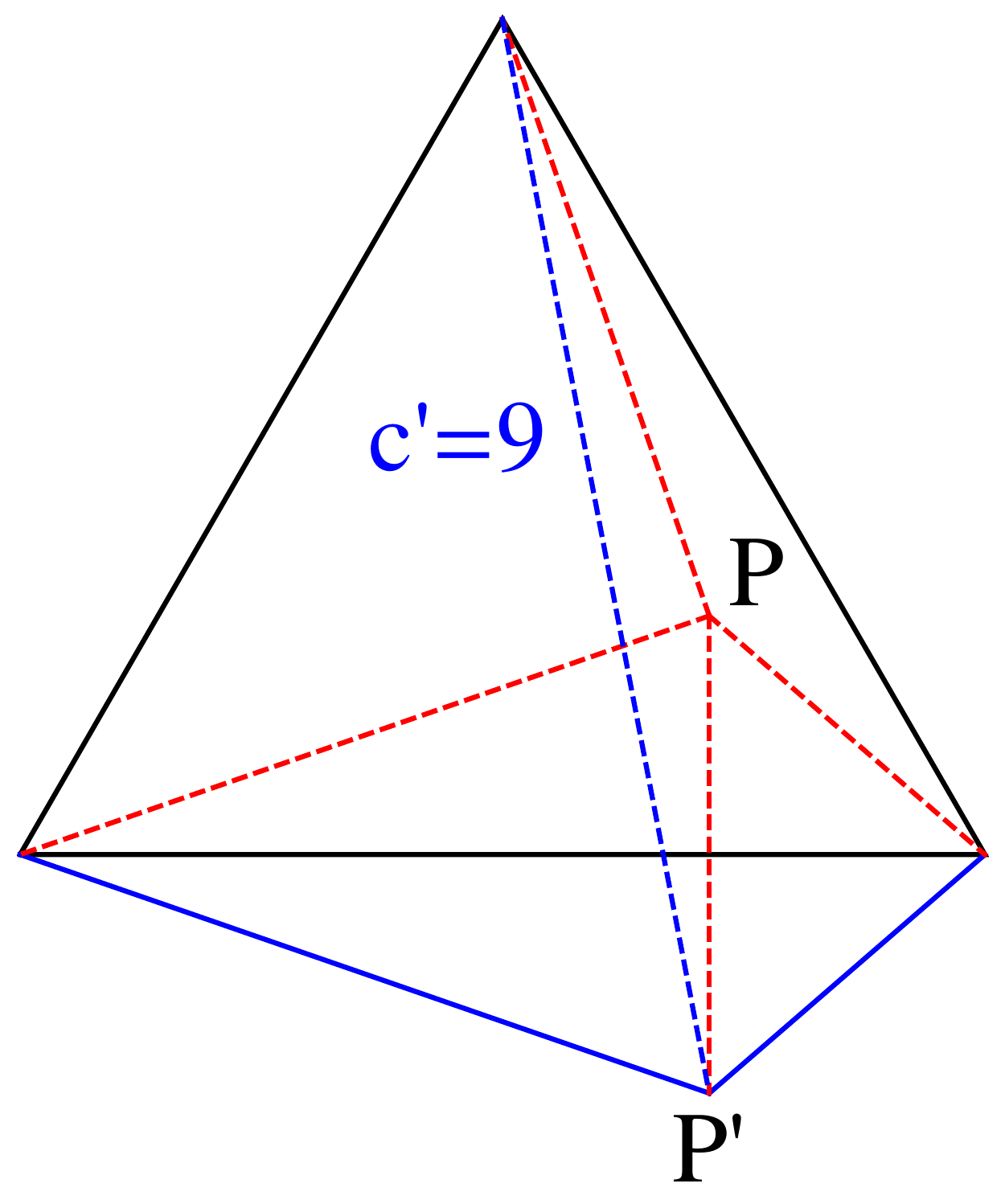}}
  \caption{The new figure is generated by reflecting two of the segments across the side of the equilateral triangle.  It is clear that if the original point is inside the triangle, then the new point must be outside the triangle and vice versa.}
  \label{operation}
\end{figure}

\begin{figure}
\centering
\includegraphics[scale=0.4]{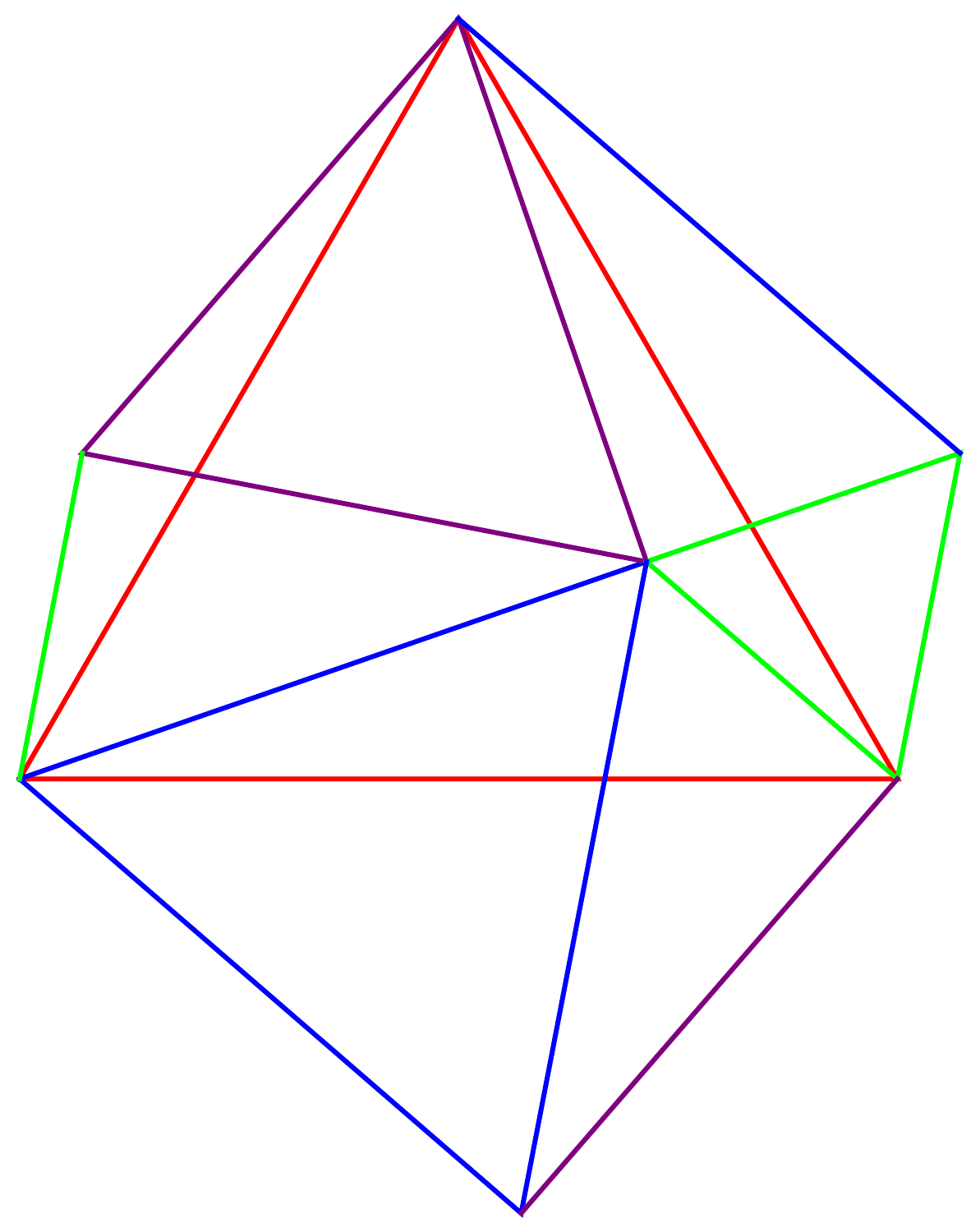}
\caption{Segments of the same length are colored the same color.}
\label{encapsulation}
\end{figure} 

\section{Reduction Theory and Root Quadruples}\label{rootquadruples}
Suppose we begin with the quadruple $(1,\, 1,\, 3,\, 4)$.  Let a reducing action be an operation that decreases the sum of the elements in a triangle quadruple.  We can reduce this to $(1,\, 1,\, 3,\, 1)$ and then to $(1,\, 1,\, 0,\, 1)$.  At this point, no further operation can reduce the sum.

\begin{definition}
A triangle quadruple $t=(a,\, b,\, c,\, d)$ is a \emph{root quadruple} if it is not possible to apply a generator that reduces the sum $a+b+c+d$.
\end{definition}  
\begin{lemma}\label{reductionalgorithm}
For any triangle quadruple $t=(a, b, c, d)$, applying the generator corresponding to the largest element does not increase the sum $a+b+c+d$.  In addition, $t$ can be reduced to a root quadruple $t'=(0,\, x,\, x,\, x)$ (or permutations) in a finite number of operations, where $x=\gcd(a,\, b,\, c,\, d)$. 
\end{lemma}

\begin{proof}
Assume that $a\geq b\geq c\geq d$ without loss of generality.  From the condition, we have $$a=\frac{1}{2}(b+c+d)\pm\sqrt{\frac{3}{2}(bc+cd+db)-\frac{3}{4}(b^2+c^2+d^2)}.$$  We first show that  $$a=\frac{1}{2}(b+c+d)+\sqrt{\frac{3}{2}(bc+cd+db)-\frac{3}{4}(b^2+c^2+d^2)}.$$  Suppose for the sake of contradiction this is not so.  If $-b+c+d<0$, then note that
\begin{align*}
&\frac{1}{2}(-b+c+d)<\sqrt{\frac{3}{2}(bc+cd+db)-\frac{3}{4}(b^2+c^2+d^2)}\\
\Rightarrow\hspace{1mm}& \frac{1}{2}(b+c+d)-\sqrt{\frac{3}{2}(bc+cd+db)-\frac{3}{4}(b^2+c^2+d^2)}<b\\
\Rightarrow\hspace{1mm}& a<b,
\end{align*} 
a contradiction.  Therefore, $-b+c+d\geq0$.  Then we have
\begin{align*}
&b^2+c^2+d^2-2bc-cd-2db=b(b-c-d)+c(c-b)+d(d-b)-cd<0\\
\Rightarrow\hspace{1mm}&b^2+c^2+d^2-2bc+2cd-2db<6(bc+cd+db)-3(b^2+c^2+d^2)\\
\Rightarrow\hspace{1mm}&\frac{1}{2}(-b+c+d)<\sqrt{\frac{3}{2}(bc+cd+db)-\frac{3}{4}(b^2+c^2+d^2)}\\
\Rightarrow\hspace{1mm}&a<b,
\end{align*}
also a contradiction.  It follows that we must have, $$a=\frac{1}{2}(b+c+d)+\sqrt{\frac{3}{2}(bc+cd+db)-\frac{3}{4}(b^2+c^2+d^2)},$$ as claimed.  Now, we will show that $2a\geq b+c+d$, with equality when $d=0$.  Note that $$2a=b+c+d+\sqrt{6(bc+cd+db)-3(b^2+c^2+d^2)}\geq b+c+d,$$ with equality when 
\begin{align*}
&(b^2+c^2+d^2)-2(bc+cd+db)=0\\
\Rightarrow \hspace{1mm}&d^2-2(b+c)d+(b^2-2bc+c^2)=0\\
\Rightarrow \hspace{1mm}&d=\frac{2(b+c)\pm\sqrt{4(b+c)^2-4(b^2-2bc+c^2)}}{2}=b+c\pm2\sqrt{bc}=(\sqrt{b}\pm\sqrt{c})^2.
\end{align*}
Assume $c\neq 0$.  If $d=(\sqrt{b}+\sqrt{c})^2$, then clearly $d>c$, contradicting the minimality of $d$.  Therefore, we must have $d=(\sqrt{b}-\sqrt{c})^2$. However, then $$a=\frac{1}{2}(b+c+d)=\frac{1}{2}(b+c+b-2\sqrt{bc}+c)=b+c-\sqrt{bc}=b-\sqrt{c}(\sqrt{b}-\sqrt{c})\leq b,$$ contradicting the maximality of $a$ unless $b=c$, in which case the triangle quadruple $(a,\, b,\, c,\, d)=(a,\, a,\, a,\, 0)$.  The case where $c=0$ generates the same solution.
\end{proof}

In addition, it is clear from this that the nonnegativity of quadruples is preserved under this operation.  Indeed, the product of $d$ and $a+b+c-d$ is $$3(a^2+b^2+c^2)-(a+b+c)^2=
 (a-b)^2+(b-c)^2+(c-a)^2,$$ the constant term of the equation $3(a^2+b^2+c^2+d^2)=(a+b+c+d)^2$.  Note that this is nonnegative, so $d(a+b+c-d)\geq 0$, implying that if $d>0$, then $a+b+c-d\geq 0$ as well.  On the other hand, if $d=0$, then $a=b=c$, as proved in Lemma~\ref{reductionalgorithm}.
  
\begin{lemma}\label{invariant}
A triangle quadruple $(a,\, b,\, c,\, d)$ can be reduced to the unique quadruple $(0,\, x,\, x,\, x)$ (or permutations), where $x=\gcd(a,\, b,\, c,\, d)$. 
\end{lemma}

\begin{proof}
Suppose we begin with the triangle quadruple $(a,\, b,\, c,\, d)$.  Let $x=\gcd(a,\, b,\, c,\, d)$.  Then it is clear that $x=\gcd(a,\, b,\, c,\, a+b+c-d)$.  Therefore, the greatest common divisor of the elements is an invariant, so by Lemma~\ref{reductionalgorithm}, by applying the generator corresponding the largest element at every step, it can be reduced to the unique quadruple $(0,\, x,\, x,\, x)$ (or permutations).  
\end{proof}

\begin{corollary}
It easily follows from Lemma~\ref{invariant} that for any two triangle quadruples $t=(a,\, b,\, c,\, d)$ and $t'=(a',\, b',\, c',\, d')$, if $\gcd(a,\, b,\, c,\, d)=\gcd(a',\, b',\, c',\, d')$, then it is possible to obtain $t'$ from $t$ by applying a finite number of generators and vice versa.  Therefore, all primitive triangle quadruples are contained in one orbit of $(0,\, 1,\, 1,\, 1)$ in $T$.  
\end{corollary}


\section{Geometric Groups}\label{geometricgroups}
\begin{theorem}
The triangle group with generators $S_1$, $S_2$, $S_3$, $S_4$ is a hyperbolic Coxeter group.
\end{theorem}

\begin{proof}
It is easy to verify that $S_1$, $S_2$, $S_3$, $S_4$ satisfy \begin{equation}\label{coxeter1}S_i^2=I\text{ for }i\in\{1,\, 2,\, 3,\, 4\}\end{equation} and \begin{equation}\label{coxeter2}(S_iS_j)^3=I\text{ for $i\neq j$,\text{ and }$i,\, j\in\{1,\, 2,\, 3,\, 4\}$.}\end{equation}  

\begin{lemma}\label{coxetergroup1}
The generators $S_i$ are reflections.  
\end{lemma}

\begin{proof}
The eigenvalues of $S_i$ are 1, 1, 1, $-1$.  It follows that the operation corresponding to $S_i$ is the reflection with respect to the hyperplane spanning the eigenvectors $v_{i_1}$, $v_{i_2}$, $v_{i_3}$ of $S_i$.
\end{proof}
\begin{lemma}\label{coxetergroup2}
For $x=(a,\, b,\, c,\, d)$, each $S_i$ preserves the quadratic form $$Q(x)=3(a^2+b^2+c^2+d^2)-(a+b+c+d)^2=xAx^T,$$ where $$A=\begin{pmatrix}2&-1&-1&-1\\-1&2&-1&-1\\-1&-1&2&-1\\-1&-1&-1&2\\\end{pmatrix}.$$  That is, $Q(x)=Q(S_ix)$. 
\end{lemma}
\begin{proof}
This is easily verified by a straightforward calculation.
\end{proof}
Construct an abstract Coxeter group with $A$ as its Cartan matrix. By Lemma~\ref{coxetergroup1} and Lemma~\ref{coxetergroup2} and application of basic properties of Coxeter groups, we conclude that the triangle group is the constructed Coxeter group. Section 5.4 in [7] proves that all relations between the matrices $S_i$ formally follow from \eqref{coxeter1} and \eqref{coxeter2}.  In other words, the geometric representation is faithful.  In particular, since the Cartan matrix $A$ of the triangle group has a $2\times 2$ positive definite submatrix and negative determinant, its signature must be $(3,\,1)$, so the triangle group is indeed hyperbolic.  
\end{proof}
Note that the cubic relations for $S_i$ are unique to the triangle group case, as they do not hold in the Apollonian circle case.  The corresponding Coxeter diagram that encodes the Coxeter matrix is a connected graph, by~\eqref{coxeter2}.
 
\section{Properties of Triangle Quadruples}\label{properties}
\subsection{The Number of Triangle Quadruples Bounded by Height}
For a triangle quadruple $Q=(a,\, b,\, c,\, d)$, define $H(Q)$ to be the height $\sqrt{a^2+b^2+c^2+d^2}$ of $Q$.

\begin{theorem}\label{height}
For an integer $n$, define $F(n)$ to be the number of triangle quadruples $Q$ with height $H(Q)\leq n$.  Then $F(n)=O(n^2\log^3(n))$. 
\end{theorem} 
\begin{proof}
Suppose that $a\geq b\geq c\geq d$ without loss of generality.  Consider the substitution $x=a$, $y=b$, $z=a+b-c$, $w=a+b-d$.  Then the quadratic form $3(a^2+b^2+c^2+d^2)-(a+b+c+d)^2$ can be rewritten as $-6xy+2z^2-2zw+2w^2$, so it suffices to show that the number of zeros of the quadratic form $-3xy+z^2-zw+w^2$ with $x^2+y^2+z^2+w^2\leq Cn^2$ for some constant $C$ is $O(n^2\log^3(n))$.  It is clear that any solution generates a nonnegative quadruple.  Indeed, if $d<0$, then $w>a+b=x+y$, so $z^2-zw+w^2=(z-w/2)^2+(3/4)w^2>(3/4)(x+y)^2\geq 3xy$, a contradiction.  In addition, the integer property is preserved because the transformation matrix $$\begin{pmatrix}1&0&0&0\\0&1&0&0\\1&1&-1&0\\1&1&0&-1\end{pmatrix}$$ has determinant 1.

Consider such a zero $(w,\, x,\, y,\, z)$, so $z^2-zw+w^2=3xy\leq Cn^2$, where $C$ is a constant.  Set $3xy=k$.  For a fixed $k$, define $A(k)$ to be the number of integer solutions to the equation $z^2-zw+w^2=k$.  In addition, for a fixed $k$, it is clear that the number of integer solutions to $3xy=k$ is bounded above by $d(k)$, the number of divisors of $k$.  

It follows that the number of zeros of the quadratic form $-3xy+z^2-zw+w^2$ with $x^2+y^2+z^2+w^2\leq Cn^2$ is bounded above by $$\sum_{k\leq Cn^2}A(k)d(k).$$

\begin{lemma}\label{A(k)}
For any positive integer $k$, $A(k)\leq 6d(k)$.
\end{lemma}

\begin{proof}
Note that the function $A(k)$ is the number of factorizations of $k$ into a product of an integer and its conjugate in the ring $\mathds{Z}[\zeta]$, where $\zeta$ is a cubic root of 1.  This ring has 6 units (the 6th roots of unity).
Define $B(m):=A(m)/6$.  For a prime $p$, if $p\equiv 2\pmod{3}$, then $p$ is still a prime in $\mathds{Z}[\zeta]$, so $B(p^j)=0$ if $j$ is odd, while $B(p^j)=1$ if $j$ is even (unique factorization $p^j=(p^{j/2})(p^{j/2})$ up to units).  If $p=3$ or if $p\equiv 1\pmod{3}$, then $p$ factors uniquely in $\mathds{Z}[\zeta]$ as $q\bar q$, so $p^j$ can only factor as $z\bar z$, where $z=q^i(\bar q)^{j-i}$.  Therefore, there are $j+1$ factorizations up to units, so $B(p^j)=j+1$.  

From these observations, it is straightforward to show that $B(m)$ is multiplicative, implying that if $m$ contains primes congruent to 2 modulo 3 raised to odd powers, then $B(m)=0$, and if not, then $B(m)=(b_1+1)\cdots(b_s+1)$ if $m$ is factored as $p_1^{2a_1}\cdots p_k^{2a_k}q_1^{b_1}\cdots q_s^{b_s}$, where $p_i$ are of the form $3r+2$ and $q_i=3$ or $q_i=3r+1$.  See \cite{Hardy}.   In our case, it suffices to show that $B(k)\leq d(k)$ for $k=p^j$, a prime power.  We have shown above that $B(p^j)=0$, 1, or $j+1$, and since $d(p^j)=j+1$, it follows that $B(p^j)\leq d(p^j)$ in all cases, as desired.  
\end{proof}

By Lemma~\ref{A(k)}, the number of zeros of the quadratic form $-3xy+z^2-zw+w^2$ with $x^2+y^2+z^2+w^2\leq Cn^2$ is bounded above by $$\sum_{k\leq Cn^2}6(d(k))^2.$$  

\begin{lemma}\label{logbounds}
For any positive integer $n$, $\displaystyle\sum_{k=1}^{n}(d(k))^2=O(n\log^{3}n)$.
\end{lemma}

\begin{proof}
Note that 
\begin{align*}
\sum_{k=1}^{n}(d(k))^2&=\sum_{k=1}^{n}{\sum_{\substack{i|k\\j|k}}1}=\sum_{\substack{1\leq i\leq n\\1\leq j\leq n}}{\sum_{\substack{k\leq n\\i|k\\j|k}}1}\leq\sum_{\substack{1\leq i\leq n\\1\leq j\leq n}}\frac{n}{\text{lcm}(i,\, j)}=n\sum_{\substack{1\leq i\leq n\\1\leq j\leq n}}\frac{1}{\text{lcm}(i,\, j)}\\
&=\sum_{d=1}^{n}\frac{n}{d}\sum_{\substack{1\leq i'\leq n/d\\1\leq j'\leq n/d\\\gcd(i',\, j')=1}}\frac{1}{i'j'}\leq\sum_{d=1}^{n}\frac{n}{d}\sum_{\substack{1\leq i'\leq n/d\\1\leq j'\leq n/d}}\frac{1}{i'j'}\\
&=\sum_{d=1}^{n}\frac{n}{d}\left(\sum_{i'=1}^{n/d}\frac{1}{i'}\right)^2\leq Cn\sum_{d=1}^{n}\frac{1}{d}\log^2\left(\frac{n}{d}\right)\\
&\leq Cn\log^{2}n\sum_{d=1}^{n}\frac{1}{d}\leq Cn\log^{3}n,
\end{align*}
as claimed.
\end{proof}

For a generalization of this result, see Formula $(1.80)$ in \cite{Iwaniec}.  Replacing $n$ by $n^2$ in Lemma~\ref{logbounds}, we conclude that the entire expression is bounded by $$\sum_{k\leq Cn^2}6(d(k))^2=O(n^2\log^3(n)),$$ as desired.   
\end{proof}

\subsection{The Number of Triangle Quadruples with Specific Elements}
Suppose we want to calculate the number of triangle quadruples containing the pair of integers $(p,\, q)$, with $p, q>0$.  Consider the substitution $x=p$, $y=q$, $z=p+q-c$, $w=p+q-d$, so the relation $3(p^2+q^2+c^2+d^2)=(p+q+c+d)^2$ can be rewritten as $z^2-zw+w^2=3pq$.  Therefore, it suffices to determine the number of integer solutions to $z^2-wz+z^2=3pq$. The number of solutions is $6B(3pq)$, where $B(m)$ is defined as in the proof of Lemma~\ref{A(k)}.  The closed formula computes the number of solutions and the solutions themselves. 

\subsection{The Number of Triangle Quadruples Bounded by Maximal Element} 
\begin{theorem}\label{maximalelement}
Let $\text{max}(Q)$ denote the maximal element in the  quadruple $Q$.  Then $|\{Q:\text{max}(Q)\leq n\}|=O(n^2\log^3(n))$.
\end{theorem}

\begin{proof}
Note that $\text{max}(Q)\leq H(Q)\leq 2\text{max}(Q)$, so by Theorem~\ref{height},  $$|\{Q:\text{max}(Q)\leq n\}|\sim|\{Q:H(Q)\leq 2n\}|=O(n^2\log^3(n))$$ as well.
\end{proof}  

\subsection{The Norm of Triangle Quadruples After $n$ Operations}
Define a word $w$ of generators $S_1$, $S_2$, $S_3$, $S_4$ to be a reduced word if it is not equal in $T$ to a word of smaller length.  Let $W$ be the set of all reduced words.  For $n=4m+i$ with $0\leq i\leq 3$, define the reduced word $R_n$ to be $R_i(S_4S_3S_2S_1)^m$, where $R_i=I,\, S_1,\, S_2S_1,\, S_3S_2S_1$ for $0\leq i \leq 3$, respectively. For a matrix $A$, $||A||_\infty$ is defined to be its maximum absolute row sum, so for a column vector $v=(a,\, b,\, c,\, d)^T$, $||v||_\infty=\text{max}(a,\, b,\, c,\, d)$.  Let $r=(0,\, x,\, x,\, x)$ be any root quadruple.   
\begin{theorem}\label{maxgrowth}
For any reduced word $w$ of length $n$, $$||wr||_\infty\leq||R_nr||_\infty=O\left(\left(\frac{1}{4}\left(7+\sqrt{13}+2\sqrt{\frac{75}{2}+\frac{21\sqrt{13}}{2}}\right)\right)^{n/4}\right).$$  
\end{theorem}
\begin{proof}
Write $w=S_{i_n}S_{i_{n-1}}\cdots S_{i_1}$.  Let $w^n=w r$ and $r^n=R_n r$.  Suppose that the elements of $w^{n}$ and $r^{n}$ are $$w_1^{(n)}\leq w_2^{(n)}\leq w_3^{(n)}\leq w_4^{(n)}\hspace{5mm}\text{and}\hspace{5mm}r_1^{(n)}\leq r_2^{(n)}\leq r_3^{(n)}\leq r_4^{(n)}.$$  We will show that \begin{equation}\label{condition1}w_i^{(n)}\leq r_i^{(n)}\text{ for } 1\leq i\leq4\end{equation} and \begin{equation}\label{condition2}w_4^{(n)}-w_1^{(n)}\leq r_4^{(n)}-r_1^{(n)}\end{equation} by induction on $n$.  The base case with $n=1$ is trivial.  For the inductive hypothesis, assume that the claims hold for $n=k-1$.  To prove that~\eqref{condition1} holds for $n=k$, note that  
\begin{alignat*}{2}
w_1^{(k)}&\leq w_2^{(k-1)}&\leq r_2^{(k-1)}&=r_1^{(k)},\\
w_2^{(k)}&\leq w_3^{(k-1)}&\leq r_3^{(k-1)}&=r_2^{(k)},\\
w_3^{(k)}&\leq w_4^{(k-1)}&\leq r_4^{(k-1)}&=r_3^{(k)},
\end{alignat*}
and
\begin{align*}
w_4^{(k)}&\leq w_2^{(k-1)}+w_3^{(k-1)}+w_4^{(k-1)}-w_1^{(k-1)}\\
&\leq w_2^{(k-1)}+w_3^{(k-1)}+(w_4^{(k-1)}-w_1^{(k-1)})\\
&\leq r_2^{(k-1)}+r_3^{(k-1)}+(r_4^{(k-1)}-r_1^{(k-1)})\\
&=r_4^{(k)}.
\end{align*}
To prove that~\eqref{condition2} holds for $n=k$, consider two cases.  If $w_1^{(k)}=w_2^{(k-1)}$, then 
\begin{align*}
w_4^{(k)}-w_1^{(k)}&=[w_2^{(k-1)}+w_3^{(k-1)}+w_4^{(k-1)}-w_1^{(k-1)}]-w_2^{(k-1)}\\
&=w_3^{(k-1)}+[w_4^{(k-1)}-w_1^{(k-1)}]\\
&\leq r_3^{(k-1)}+[r_4^{(k-1)}-r_1^{(k-1)}]\\
&=[r_2^{(k-1)}+r_3^{(k-1)}+r_4^{(k-1)}-r_1^{(k-1)}]-r_2^{(k-1)}\\
&=r_4^{(k)}-r_1^{(k)}.
\end{align*}
If $w_1^{(k)}=w_1^{(k-1)}$, then 
\begin{align*}
w_4^{(k)}-w_1^{(k)}&\leq [w_1^{(k-1)}+w_3^{(k-1)}+w_4^{(k-1)}-w_2^{(k-1)}]-w_1^{(k-1)}\\
&\leq[w_2^{(k-1)}+w_3^{(k-1)}+w_4^{(k-1)}]-w_1^{(k-1)}-w_2^{(k-1)}\\
&\leq [r_2^{(k-1)}+r_3^{(k-1)}+r_4^{(k-1)}-r_1^{(k-1)}]-r_2^{(k-1)}\\
&=r_4^{(k)}-r_1^{(k)}.
\end{align*} 
This completes our induction.  Note that the proof is easily derived from the proof of Theorem 7.1 in \cite{Graham2} with a few modifications.  

To bound $||R_nr||_\infty$, consider the characteristic polynomial of $S_4S_3S_2S_1$, $1-7t-15t^2-7t^3+t^4$.  Note that the growth of $(S_4S_3S_2S_1)^m$ is bounded by $\gamma^{m}$, where $$\gamma=\frac{1}{4}\left(7+\sqrt{13}+2\sqrt{\frac{75}{2}+\frac{21\sqrt{13}}{2}}\right)\approx 8.795$$ is the largest eigenvalue of $S_4S_3S_2S_1$ and the largest root of its characteristic polynomial.   
\end{proof}

\subsection{The Number of Triangle Quadruples Generated After $n$ Operations}
\begin{theorem}\label{numberoftrianglequadruples}
Let $W_n$ be the set of reduced words with length not greater than $n$.  Then $|W_n|=O(\lambda^n)$, where $\lambda=\frac{1+\sqrt{13}}{2}$.
\end{theorem}
\begin{proof}
By \cite{Albar}, the growth coefficients $G_n$ of the triangle group, computing the number of reduced words with length $n$ generating distinct matrices, are generated by the recurrence $G_0=1$, $G_1=4$, $G_2=12$, and $G_n=2G_{n-1}+2G_{n-2}-3G_{n-3}$ for $n\geq 3$.  The solution for this recurrence is $$G_n=c_1\left(\frac{1-\sqrt{13}}{2}\right)^n+c_2\left(\frac{1+\sqrt{13}}{2}\right)^n=O\left(\left(\frac{1+\sqrt{13}}{2}\right)^n\right).$$  Therefore, if we let $\lambda$ denote $\frac{1+\sqrt{13}}{2}$, then $G_n\sim \lambda^n$.  It follows that $$|W_n|=G_0+G_1+G_2+\cdots+G_n\sim 1+\lambda+\lambda^2+\cdots+\lambda^n=\frac{\lambda^{n+1}-1}{\lambda-1}=O(\lambda^n),$$ as desired.
\end{proof}

\begin{theorem}\label{quadruplesandwords}
For any root quadruple $r$ of the form $(0,\, x,\, x,\, x)$, $c_1\lambda^n/n^2\leq |W_nr|\leq c_2\lambda^n$ for some constants $c_1$ and $c_2$.  That is, $|W_nr|$ and $|W_n|$ differ asymptotically by a factor that is $O(n^2)$.
\end{theorem}
\begin{proof}
The upper bound is trivial by Theorem~\ref{numberoftrianglequadruples}.  To prove the lower bound, note that the stabilizer of $r$ in $T$ is $T_0:={\bold S}_3\ltimes \mathds{Z}^2$ (where ${\bold S}_3$ denotes the symmetric group in three items), the subgroup generated by $S_2$, $S_3$, $S_4$. Therefore, if $w$ and $w'\in W_n$ are in the same coset of $T_0$, then $w'=wu$ for some $u\in T_0$. It follows that $u=w^{-1}w'$, so the length of $u$ is at most $2n$.  But by \cite{Albar}, the growth series of $T_0$, which has three generators, is $\frac{1+t+t^2}{(1-t)^2}$, from which it can be computed that the growth coefficients $G_n$ are generated by $G_0=1$ and $G_n=3n$ for $n>0$.  Therefore, the number of elements in $T_0$ of length at most $2n$ is $$1+3(1+2+\cdots+2n)=6n^2+3n+1\leq kn^2,$$ for some constant $k$. For any $c\in W_nr$, define $N(c):=|\{w\in W_n:wr=c\}|$.  Then $$|W_n|=\sum_{c\in W_nr}N(c)\leq|W_nr|kn^2,$$ so $|W_nr|\geq |W_n|/kn^2$, which implies the lower bound.

\end{proof}

\subsection{Elements in Triangle Quadruples}
Let $O(4):=O(4,\, A,\, \mathds{C})$, the group of matrices that preserve the quadratic form $A$.  Similarly, let $SO(4):=SO(4,\, A,\, \mathds{C})$.
\begin{theorem}
For every triangle quadruple $Q=(a,\, b,\, c,\, d)$, define $\alpha(Q)$ to be the number of prime factors (with multiplicities) in $a\cdot b\cdot c\cdot d$.  There are infinitely many $Q$ such that $\alpha(Q)\leq k$ for some constant $k$. 
\end{theorem}
\begin{proof}
We will invoke Theorem 3 from \cite{Golsefidy}, which states that the result is true if the Zariski
 closure $\overline{T}$ of $T$ in the complex algebraic group $GL(4,\, \mathds{C})$ is a semisimple algebraic group.  By this theorem, it suffices to verify that $\overline{T}$ is semisimple.
 
\begin{lemma}\label{zariskiclosure}
The Zariski closure of the triangle group, $\overline{T}$, is $O(4)$.
\end{lemma}
\begin{proof}
Consider the subgroup $T_{123}={\bold S}_3\ltimes \mathds{Z}^2$ (where ${\bold S}_3$ denotes the symmetric group in three items) generated by $S_1$, $S_2$, $S_3$, containing the translation matrix $$A_1=S_1S_2S_1S_3=\begin{pmatrix}2&1&-2&4\\1&2&-2&4\\1&1&-1&1\\0&0&0&1\end{pmatrix}.$$  It is straightforward to verify that $$A_1^{n}=\begin{pmatrix}1+n&n&-2n&3n^2+n\\n&n+1&-2n&3n^2+n\\n&n&1-2n&6n^2-2n\\0&0&0&1\end{pmatrix}.$$  Let $$B_1={d\over dn}\bigg|_{n=0}A_1^n=\begin{pmatrix}1&1&-2&1\\1&1&-2&1\\1&1&-2&-2\\0&0&0&0\end{pmatrix}.$$  

It follows that $B_1\in\text{Lie}(\overline{T})$, along with the conjugates of $B_1$ by $T$ and their brackets.  It is straightforward to verify that the six matrices 
$$S_1B_1S_1^{-1}=\left(\begin{rsmallmatrix}-1&2&-1&-1\\-1&2&-1&2\\-1&2&-1&-1\\0&0&0&0\end{rsmallmatrix}\right),\hspace{1mm}S_2B_1S_2^{-1}=\left(\begin{rsmallmatrix}2&-1&-1&2\\2&-1&-1&-1\\2&-1&-1&-1\\0&0&0&0\end{rsmallmatrix}\right),$$
$$(S_1B_1S_1^{-1})(S_4B_1S_4^{-1})-(S_4B_1S_4^{-1})(S_1B_1S_1^{-1})=\left(\begin{rsmallmatrix}3&-6&12&-6\\12&3&-6&-6\\-6&12&3&-6\\0&0&0&-9\end{rsmallmatrix}\right),$$
$$(S_2B_1S_2^{-1})(S_4B_1S_4^{-1})-(S_4B_1S_4^{-1})(S_2B_1S_2^{-1})=\left(\begin{rsmallmatrix}3&12&-6&-6\\-6&3&12&-6\\12&-6&3&-6\\0&0&0&-9\end{rsmallmatrix}\right),$$
$$[S_1S_4B_1(S_1S_4)^{-1}](S_4B_1S_4^{-1})-(S_4B_1S_4^{-1})[S_1S_4B_1(S_1S_4)^{-1}]=\left(\begin{rsmallmatrix}30&12&12&-24\\12&-6&-6&-6\\12&-6&-6&-6\\36&-18&-18&-18\end{rsmallmatrix}\right),$$
$$[S_2S_4B_1(S_2S_4)^{-1}](S_4B_1S_4^{-1})-(S_4B_1S_4^{-1})[S_2S_4B_1(S_2S_4)^{-1}]=\left(\begin{rsmallmatrix}-6&12&-6&-6\\12&30&12&-24\\-6&12&-6&-6\\-18&36&-18&-18\end{rsmallmatrix}\right)$$ are linearly independent.  These six linearly independent elements of $\text{Lie}(\overline{T})$ determine a basis, implying that they span $\text{Lie}(SO(4))$ (with dimension 6), which equals $\text{Lie}(O(4))$, so $\text{Lie}(O(4))\subset\text{Lie}(\overline{T})$.  In addition, by Lemma~\ref{coxetergroup2}, $\overline{T}\subset O(4)$, so $\text{Lie}(\overline{T})\subset\text{Lie}(O(4))$.  Therefore, $\text{Lie}(\overline{T})=\text{Lie}(O(4))$.  This implies that $\overline{T}$ is either $SO(4)$ or $O(4)$.  However, $\overline{T}$ cannot be $SO(4)$ because $T$ contains matrices of determinant $-1$ ($S_1$, for example), so $\overline{T}=O(4)$, as desired.
\end{proof}
It is well-known that $O(4)$ is semisimple, and this completes our proof. 
\end{proof}  

\section{Extension to Higher Dimensions}\label{higherdimensions}
Lemma (a) in \cite{Gregorac} proves the following result, which extends the equation for triangle quadruples to higher dimensions. Here, we present an alternative proof. 
\begin{theorem}
For a simplex $T$ in $n$ dimensions and an arbitrary point $P$, if we let $a_0,\, a_1,\, \ldots,\, a_{n+1}$, denote the square of the side length of $T$ and the squares of the distances from $P$ to the vertices of $T$, then $a_0,\, a_1,\, \ldots,\, a_{n+1}$ satisfy $(n+1)(a_0^2+a_1^2+\cdots +a_{n+1}^2)=(a_0+a_1+\cdots +a_{n+1})^2$.  
\end{theorem}

\begin{proof}
Consider a point $P$ inside a simplex in $n$ dimensions, configured so that $P$ is at the origin. Then the vertices of the tetrahedron are the vectors $v_1,\, \ldots,\, v_{n+1}$. Set $a_0=|v_i-v_j|^2$ and $a_i=|v_i|^2$, so $2\langle v_i,\, v_j\rangle=a_i+a_j-a_0$. Since $v_i\in\mathds{R}^n$ and there are $n+1$ of them, they are linearly dependent, which implies that the Gram determinant (the determinant of the matrix with entries $\langle v_i,\, v_j\rangle$) is zero.  Therefore, $\det(G)=0$, where $G_{ij}=a_i+a_j-a_0$, for $i,\, j=1, \ldots,\, n+1$, $i\neq j$, and $G_{ii}=2a_i$. Let ${\bold a}$ be the vector with entries $a_i$, for $i=1,\, \ldots,\, n+1$, and let ${\bold u}$ be the vector whose entries all are 1. Then, for any column vector $x=(x_1,\, \ldots,\, x_{n+1})\in\mathds{R}^{n+1}$, we have $(G-a_0I)x=\langle {\bold u},\, x\rangle {\bold a}+\langle {\bold a}-a_0{\bold u},\, x\rangle {\bold u}$.  Since the image of $G-a_0I$ has dimension 2, it follows that $G-a_0I$ has rank 2 and consequently, an $n-1$-dimensional kernel $K$. Set $e_1={\bold a}$, $e_2={\bold u}$, and select a basis $e_3,\, \ldots,\, e_{n+1}$ of $K$. In this basis, all the rows of the matrix $G-a_0I$ are zero, except the first two, and the upper left $2\times 2$ block contains the elements
$$\begin{matrix}
\langle {\bold u},\, {\bold a}\rangle &\langle {\bold u},\, {\bold u}\rangle\\
\langle {\bold a}-a_0{\bold u},\, {\bold a}\rangle &\langle {\bold a}-a_0{\bold u},\, {\bold u}\rangle.
\end{matrix}$$
The characteristic polynomial of $G-a_0I$ is thus the characteristic polynomial of this $2\times 2$ matrix multiplied by $z^{n-1}$, or $$(z^2-\langle 2{\bold a}-a_0{\bold u},\, {\bold u}\rangle z+\langle {\bold u},\, {\bold a}\rangle\langle {\bold a}-a_0{\bold u},\, {\bold u}\rangle-\langle {\bold u},\, {\bold u}\rangle\langle {\bold a}-a_0{\bold u},\, {\bold a}\rangle)z^{n-1}.$$  The characteristic polynomial of $G$ is obtained from this by replacing $z$ by $z-a_0$, or $$((z-a_0)^2-\langle 2{\bold a}-a_0{\bold u},\, {\bold u}\rangle z+\langle {\bold u},\, {\bold a}\rangle\langle {\bold a}-a_0{\bold u},\, {\bold u}\rangle-\langle {\bold u},\, {\bold u}\rangle\langle {\bold a}-a_0{\bold u},\, {\bold a}\rangle)(z-a_0)^{n-1}.$$
Therefore, substituting $z=0$, we calculate 
\begin{align*}
\det(G)&=(-1)^{n+1}\det(-G)\\
&=(-1)^{n+1}(-a_0)^{n-1}(a_0^2+a_0\langle 2{\bold a}-a_0{\bold u},\, {\bold u}\rangle +\langle {\bold u},\, {\bold a}\rangle \langle {\bold a}-a_0{\bold u},\, {\bold u}\rangle-\langle{\bold u},\, {\bold u}\rangle\langle {\bold a}-a_0{\bold u},\, {\bold a}\rangle)\\
&=a_0^{n-1}(a_0^2+a_0\langle 2{\bold a}-a_0{\bold u},\, {\bold u}\rangle +\langle {\bold u},\, {\bold a}\rangle \langle {\bold a}-a_0{\bold u},\, {\bold u}\rangle-\langle{\bold u},\, {\bold u}\rangle\langle {\bold a}-a_0{\bold u},\, {\bold a}\rangle)\\
&=a_0^{n-1}[(a_0+a_1+\cdots+a_{n+1})^2-(n+1)(a_0^2+a_1^{2}+\cdots+a_{n+1}^2)].
\end{align*}
Since $\det(G)=0$, we conclude $$(a_0+a_1+\cdots+a_{n+1})^2-(n+1)(a_0^2+a_1+\cdots+a_{n+1}^2)=0,$$ as desired.
\end{proof}
 
Note that the general relation simplifies to $$na_{n+1}^2-2a_{n+1}\sum_{i=0}^{n}a_i+(n+1)\sum_{i=0}^{n}a_i^2-\left(\sum _{i=0}^{n}a_i\right)^2=0,$$ so we can reduce as follows $$(a_0,\, a_1,\, \ldots,\, a_{n+1})\rightarrow(a_0,\, a_1,\, \ldots,\, \frac{2}{n}\sum_{i=0}^{n}a_i-a_{n+1}).$$  Notice that in dimensions $n>2$, the analogous operations do not preserve the integer property of the elements.

\section{Open Questions}\label{openquestions}
1. Beginning with a specific root quadruple, is it possible to calculate the asymptotics of the average value of the maximum element in the triangle quadruple obtained after $n$ operations as $n$ goes to infinity?\\
\\
This problem follows naturally from our discussions of the number of triangle quadruples characterized by maximal elements and by height, as well as the growth rate of the triangle group.  Intuitively, the average value of the maximal element after $n$ operations should grow exponentially.\\
\\ 
2. Given a triangle quadruple, is it possible to calculate how many reduced words of minimal possible length can generate it?\\
\\
This is related to the previous question in that it requires an analysis of the properties of triangle quadruples after $n$ operations.

\section{Acknowledgments}
We thank Prof.~R.~Stanley at MIT for suggesting this project and for answering our questions, as well as various other professors at MIT for discussing aspects of the project.  We also thank Prof.~P.~Etingof for assistance in proving the theorems in Section~\ref{properties} and both him and Dr.~T.~Khovanova for reviewing this paper and sharing their insights.  Finally, we thank the PRIMES program, without which this project would not have been possible.

\end{document}